\date{}
\documentclass{article}

\usepackage{amsmath,amsfonts,amssymb,amsthm}
\usepackage{latexsym, xspace, enumerate}
\usepackage[mathscr]{eucal}
\usepackage[all]{xypic}
\usepackage{hyperref}
\usepackage{amscd}
\usepackage{vmargin}

\setmarginsrb{28mm}{20mm}{28mm}{25mm}%
          {10mm}{7mm}{10mm}{10mm}

\newtheorem{theorem}{Theorem}[section]

\newtheorem{proposition}[theorem]{Proposition}

\newtheorem{lemma}[theorem]{Lemma}
\newtheorem{fact}[theorem]{Fact}
\newtheorem{corollary}[theorem]{Corollary}

\newtheorem{definition}[theorem]{Definition}
\theoremstyle{definition}
\newtheorem{question}[theorem]{Question}
\newtheorem{example}[theorem]{Example}
\newtheorem{remark}[theorem]{Remark}

\newcommand{\R}{\mathbb R}

\newcommand{\Q}{\mathbb Q}

\newcommand{\FIAT}{\mathrm{fiAT}}
\newcommand{\AT}{\mathrm{AT}}

\numberwithin{equation}{section}

\title{The Bridge Theorem for totally disconnected LCA groups}

\author{Dikran Dikranjan
\\{\footnotesize {\tt  dikran.dikranjan@uniud.it}} 
\\{\footnotesize Dipartimento di Matematica e Informatica,}
\\{\footnotesize Universit\`{a} di Udine,}
\\{\footnotesize Via delle Scienze, 206 - 33100 Udine, Italy} 
 \and Anna Giordano Bruno
\\{\footnotesize {\tt  anna.giordanobruno@uniud.it}} 
\\{\footnotesize Dipartimento di Matematica e Informatica,}
\\{\footnotesize Universit\`{a} di Udine,}
\\{\footnotesize Via delle Scienze, 206 - 33100 Udine, Italy}
 }

\begin{document}

\maketitle

\begin{abstract}
For a totally disconnected locally compact abelian group, we prove that the topological entropy of a continuous endomorphism coincides with the algebraic entropy of the dual endomorphism with respect to the Pontryagin duality. Moreover, this result is extended to all locally compact abelian groups under the assumption of additivity with respect to some fully invariant subgroups for both the topological and the algebraic entropy.
\end{abstract}

\bigskip
\noindent\textrm{\small{Key words: topological entropy, algebraic entropy, totally disconnected group, locally compact abelian group, continuous endomorphism, Pontryagin duality\\
2010 AMS Subject Classification: primary 37B40, 22B05, 22D40; secondary 20K30, 54H11, 54H20.}}

\section{Introduction}

In \cite{B} Bowen defined the topological entropy for uniformly continuous self-maps of metric spaces, and this notion was extended by Hood in \cite{hood} for uniformly continuous self-maps of uniform spaces. For continuous self-maps of compact metric spaces this topological entropy coincides with the one previously introduced by Adler, Konheim and McAndrew in \cite{AKM} for continuous self-maps of compact spaces (see \cite[Corollary 2.14]{DSV}).
Hood's extension of Bowen's definition of topological entropy applies in particular to continuous endomorphisms of locally compact abelian groups, and can be given in the following way, as explained in detail in \cite{DSV}. For a locally compact abelian group $G$, denote by $\mathcal C(G)$ the family of compact neighborhoods of $0$ in $G$. Let $\mu$ be a Haar measure on $G$ and $\phi:G\to G$ a continuous endomorphism. For every $U\in\mathcal C(G)$ and every positive integer $n$, the \emph{$n$-th $\phi$-cotrajectory of $U$} is
$$C_n(\phi,U)=U\cap\phi^{-1}(U)\cap\ldots\cap \phi^{-n+1}(U).$$
The \emph{topological entropy} of $\phi$ with respect to $U$ is
\begin{equation}\label{Htop}
H_{top}(\phi,U)=\limsup_{n\to\infty}\frac{-\log\mu(C_n(\phi,U))}{n}
\end{equation}
and it does not depend on the choice of the Haar measure $\mu$ on $G$.
The \emph{topological entropy} of $\phi$ is
$$h_{top}(\phi)=\sup\{H_{top}(\phi,U):U\in\mathcal C(G)\}.$$

\bigskip
Using ideas briefly sketched in \cite{AKM}, Weiss developed in \cite{W} the definition of algebraic entropy for endomorphisms of torsion abelian groups (this case was further developed in \cite{DGSZ}). 
Then Peters modified this definition in \cite{Pet} for automorphisms of arbitrary abelian groups, and his approach was extended to all endomorphisms in \cite{DG2}.
In \cite{Pet1} Peters gave a further generalization of the entropy he defined in \cite{Pet}, using Haar measure, for topological automorphisms of locally compact abelian groups. 
Finally, in \cite{V} Virili modified Peters' definition, in the same way as done in \cite{DG2} for the discrete case, obtaining a new notion of algebraic entropy for continuous endomorphisms $\phi$ of locally compact abelian groups $G$.
For every $U\in\mathcal C(G)$ and every positive integer $n$, the \emph{$n$-th $\phi$-trajectory of $U$} is
$$T_n(\phi,U)=U+\phi(U)+\ldots+ \phi^{n-1}(U).$$
The \emph{algebraic entropy of $\phi$ with respect to $U$} is
\begin{equation}\label{Halg}
H_{alg}(\phi,U)=\limsup_{n\to\infty}\frac{\log\mu(T_n(\phi,U))}{n};
\end{equation}
as for the topological entropy, it is proved and easy to see that it does not depend on the choice of the Haar measure $\mu$ on $G$.
The \emph{algebraic entropy} of $\phi$ is
$$h_{alg}(\phi)=\sup\{H_{alg}(\phi,U):U\in\mathcal C(G)\}.$$

\smallskip
Let $\mathcal B(G)$ be the subfamily of $\mathcal C(G)$ consisting of all compact open subgroups of $G$. 
It is worth noting that when $U\in \mathcal B(G)$ the use of Haar measure in the definitions in \eqref{Htop} and \eqref{Halg} can be avoided and the limits superior become actually limits: 

\begin{fact}\label{no-mu}\emph{\cite{DG-islam}}
Let $G$ be a locally compact abelian group, $\phi:G\to G$ a continuous endomorphism and $U\in\mathcal B(G)$. Then:
$$H_{top}(\phi,U)=\lim_{n\to \infty}\frac{\log[U:C_n(\phi,U)]}{n}\ \text{and}\ H_{alg}(\phi,U)=\lim_{n\to\infty}\frac{\log[T_n(\phi,U):U]}{n}.$$
\end{fact}

\smallskip
It follows from these formulas, that $H_{top}(\phi,\{0\})= 0$ in case $G$ is discrete, and $H_{alg}(\phi,G)= 0$ in case $G$ is compact. 
For more details on the properties of the topological and the algebraic entropy, as well as for their relations with other topics from geometric group theory, number theory and ergodic theory, see the already cited papers and also the survey papers \cite{DG-islam,DG_PC}.

\bigskip
Since its origin, the algebraic entropy was introduced in connection to the topological entropy by means of Pontryagin duality.
For a locally compact abelian group $G$ we denote its Pontryagin dual group by $\widehat G$ and for $\phi:G\to G$ a continuous endomorphism $\widehat\phi:\widehat G\to \widehat G$ is the dual endomorphism of $\phi$. This defines a contravariant functor $\;\;\widehat{}:\mathcal L\to \mathcal L$
(the Pontryagin duality functor), where $\mathcal L$ is the category of all locally compact abelian groups. In Section \ref{pd-sec} we give the basic definitions concerning Pontryagin duality and known related properties needed in this paper.

For a subcategory $\mathfrak X$ of $\mathcal L$ we denote by $\widehat{\mathfrak X}$ the image of the subcategory $\mathfrak X$ under the Pontryagin duality functor (i.e., $\widehat{\mathfrak X}$ has as objects and morphisms all $\widehat X$ and $\widehat f$,  when $X$ and $f$ run over the objects and morphisms of $\mathfrak X$, respectively). Note that $\widehat{\mathcal L}=\mathcal L$. Moreover, the subcategory $\mathfrak X$ is full if and only if the subcategory $\widehat{\mathfrak X}$ is full. 

Of particular interest are the non-full subcategories defined as follows. For a full subcategory $\mathfrak X$ of $\mathcal L$ let $\mathfrak X_{iso}$ be the subcategory of $\mathcal L$ having as objects all objects of the subcategory $\mathfrak X$ and having as morphisms only the isomorphisms in $\mathfrak X$. 

Here is a list of the full subcategories of $\mathcal L$ that frequently appear in the sequel: 
\begin{itemize}
\item $\mathfrak C$ has as objects all compact groups of $\mathcal L$;
\item $\mathfrak D$ has as objects all discrete groups of $\mathcal L$;
\item $\mathfrak Ctd$ has as objects all compact totally disconnected groups of $\mathcal L$;
\item $\mathfrak Cmet$ has as objects all compact metrizable groups of $\mathcal L$;
\item $\mathcal L td$ has as objects all totally disconnected groups of $\mathcal L$;
\end{itemize}
It is a well known fact that $$\widehat{\mathfrak C}=\mathfrak D,\ \ \widehat{\mathfrak Cmet}=\mathfrak Dcnt\ \ \text{and}\ \ \widehat{\mathfrak Ctd}=\mathfrak Dtor,$$ where $\mathfrak Dtor$ and $\mathfrak Dcnt$ are the full subcategories of $\mathfrak D$ with objects all torsion groups and all countable groups, respectively. Moreover, $$\widehat{\mathfrak Ltd}=\mathfrak Lcc,$$ where $\mathcal L cc$ has as objects all compactly covered groups of $\mathcal L$; we say that a topological group $G$ is compactly covered if each element of $G$ is contained in some compact subgroup of $G$. 

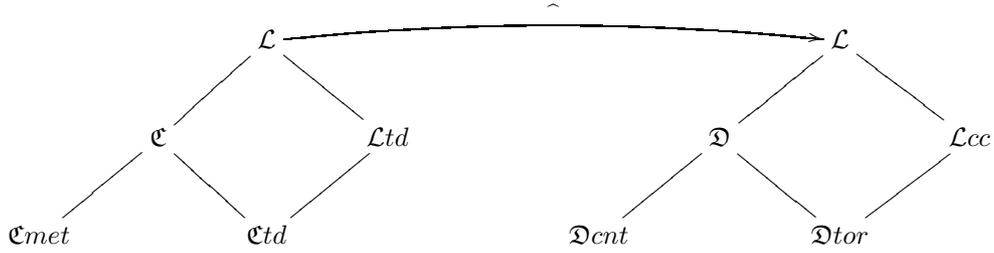
\begin{equation*}
\xymatrix{
&& \mathcal L\ar@/^/[rrrrr]^{\widehat{}} \ar@{-}[dr]\ar@{-}[dl] & & & && \mathcal L \ar@{-}[dr]\ar@{-}[dl] & \\
&\mathfrak C \ar@{-}[dr] \ar@{-}[dl] & & \mathcal Ltd \ar@{-}[dl] & && \mathfrak D \ar@{-}[dl] \ar@{-}[dr] & & \mathcal Lcc \ar@{-}[dl]\\
\mathfrak Cmet&&\mathfrak Ctd& & & \mathfrak Dcnt & &\mathfrak Dtor&
}
\end{equation*}

\medskip


\begin{definition}\label{Def_BT}
For a subcategory $\mathfrak X$ of the category $\mathcal L$ of all locally compact abelian groups, we say that $\mathfrak X$ \emph{satisfies the Bridge Theorem} \phantom{$\!\!\! \!\! \!\!\!|^{|^|}$} if $h_{top}(\phi)=h_{alg}(\widehat\phi)$ holds for every endomorphism $\phi:G\to G$ in $\mathfrak X$. 
\end{definition}

\smallskip
As we saw above, $h_{top}$ vanishes on $\mathfrak D$, while $h_{alg}$ vanishes on $\mathfrak C$ (this a consequence of Lemma \ref{base}). Hence, the subcategory $\mathfrak D$ of $\mathcal L$ satisfies the Bridge Theorem.

Weiss proved in \cite{W} that 
$\mathfrak Ctd$ satisfies the Bridge Theorem.
%
%
%
Moreover, Peters proved the same conclusion in \cite{Pet} for the category $\mathfrak Cmet_{iso}$.
Weiss Bridge Theorem and Peters Bridge Theorem were recently extended by the following theorem from \cite{DG-bt} to the whole category $\mathfrak C$,
in particular eliminating the hypothesis of total disconnectedness in Weiss Bridge Theorem.
%
%
\begin{theorem}\emph{\cite{DG-bt}}\label{DG-bt}
The category $\mathfrak C$ satisfies the Bridge Theorem. 
\end{theorem}

In this paper, and in particular in Theorem \ref{BT} below, we extend Weiss Bridge Theorem in another direction, that is relaxing the condition of compactness to local compactness. 
%
%
Indeed, in Section \ref{entropy-sec} we see that if $\phi:G\to G$ is a morphism in $\mathcal L td$,  the computation of $h_{top}(\phi)$ can be limited to taking the supremum of $H_{top}(\phi,U)$ with $U$ ranging in $\mathcal B(G)$, and respectively for $h_{alg}(\widehat\phi)$ it suffices to take the supremum of $H_{alg}(\widehat\phi,V)$ with $V$ in $\mathcal B(\widehat G)$. 
Applying Fact \ref{no-mu}, we find the precise relation between $H_{top}(\phi,U)$ and $H_{alg}(\widehat\phi,V)$ (see Proposition \ref{C_n-T_n*}), that permits to prove the following result in Section \ref{bt-sec}.

\begin{theorem}\label{BT} 
The subcategory $\mathcal L td$ of $\mathcal L$ satisfies the Bridge Theorem. 
\end{theorem}

We leave open the problem of the validity of the Bridge Theorem in the general case:

\begin{question}\label{Ques}
Which subcategories $\mathfrak X$ of $\mathcal L$ satisfy the Bridge Theorem?
\end{question}

Note that a positive answer is known in some particular cases, for example when $\mathfrak X = \mathfrak V _K$ is the full subcategory of all finite dimensional vector spaces over the locally compact field $K$, where $K$ is either $\R$ or $\Q_p$, for some prime $p$ (see Example \ref{Rn} below and \cite{DG-islam}, respectively).

\medskip
We see below that a positive answer is possible under the assumption that both the topological and the algebraic entropy are additive with respect to some fully invariant closed subgroups. To be more precise we introduce the following terminology. 

If $h$ is one of $h_{top}$ or $h_{alg}$, we say that \emph{the Addition Theorem holds for $h$} in the full subcategory $\mathfrak X$ of the category $\mathcal L$, if 
\begin{equation}\label{ATeq}
h(\phi)=h(\phi\restriction_H)+h(\overline\phi),
\end{equation}
for every endomorphism $\phi:G\to G$ in $\mathfrak X$, and every $\phi$-invariant closed subgroup $H$ of $G$, such that both 
$\phi\restriction_H$ and the continuous endomorphism $\overline\phi:G/H\to G/H$ induced by $\phi$ belong to $\mathfrak X$.
We denote briefly by $\AT_{top}$ (respectively, $\AT_{alg}$) the statement ``the Addition Theorem holds for $h_{top}$" (respectively, ``the Addition Theorem holds for $h_{alg}$"), and we do not specify $\mathfrak X$ when it is the whole $\mathcal L$.

Moreover, we denote by $\FIAT_{top}$ and $\FIAT_{alg}$ the fact that \eqref{ATeq} holds for every pair $G$, $\phi$ as above, and for every fully invariant closed subgroup $H$ of $G$. Clearly, $\AT_{top}$ and $\AT_{alg}$ are much stronger than $\FIAT_{top}$ and $\FIAT_{alg}$ respectively.

\medskip
It is known that $\AT_{top}$ holds in $\mathfrak C$ (see \cite{B,Y}), while $\AT_{alg}$ holds in $\mathfrak D$ (see \cite{DG2}). Moreover, we would like to underline that additivity is one of the most important properties of the topological and the algebraic entropy.

\medskip
Assuming that both $\FIAT_{top}$ and $\FIAT_{alg}$ hold for all locally compact abelian group, one can prove that also the Bridge Theorem holds in general. Indeed, we have the following

\begin{theorem}\label{AT->BT}
Assume that $\FIAT_{top}$ and $\FIAT_{alg}$ hold. Then $\mathcal L$ satisfies the Bridge Theorem.
\end{theorem}

The proof of this theorem is given in Section \ref{at-sec}. It relies on a special weaker form of $\FIAT$, that is, with respect to two particular fully invariant subgroups (see Remark \ref{LastRemark} for more details).

\medskip
So one should try to answer the following questions on the general validity of the Addition Theorem in $\mathcal L$.

\begin{question}
\begin{itemize}
\item[(a)] Do $\AT_{top}$ and $\AT_{alg}$ hold?
\item[(b)] Do $\FIAT_{top}$ and $\FIAT_{alg}$ hold?
\end{itemize}
\end{question}

One could try to answer this question first for some subcategory $\mathfrak X$ of $\mathcal L$, and then pass to the general case of the whole category $\mathcal L$. Another case to consider first should be that of very special subgroups, eventually starting from the ones needed in the proof of Theorem \ref{AT->BT}.

\medskip
By the time this paper was accepted, S. Virili \cite{V-BT} obtained an impressive generalization of Peters' version of the Bridge Theorem, pushing much further the technique based on harmonic analysis adopted in \cite{Pet1}. In particular, 
this gives an important contribution towards Question \ref{Ques} by proving that $\mathcal L_{iso}$ satisfies the Bridge Theorem.

\section{Topological entropy and algebraic entropy}\label{entropy-sec}

Let $G$ be a locally compact abelian group and $\phi:G\to G$ a continuous endomorphism. It is clear from the definition that the map $H_{top}(\phi,-)$ is antimonotone, that is,
\begin{center}
if $U,V\in\mathcal C(G)$ and $U\subseteq V$, then $H_{top}(\phi,U)\geq H_{top}(\phi,V)$.
\end{center}
On the other hand, the map $H_{alg}(\phi,-)$ is monotone, in other words
\begin{center}
if $U,V\in\mathcal C(G)$ and $U\subseteq V$, then $H_{alg}(\phi,U)\leq H_{alg}(\phi,V)$.
\end{center}
From these properties it follows that, in order to compute $H_{top}(\phi,-)$ and $H_{alg}(\phi,-)$, it suffices to consider respectively a local base at $0$ contained in $\mathcal C(G)$ and a cofinal subfamily of $\mathcal C(G)$. This proves the following:

\begin{lemma}\label{base}
Let $G$ be a locally compact abelian group and $\phi:G\to G$ a continuous endomorphism.
\begin{itemize}
\item[(a)] If $\mathcal B\subseteq \mathcal C(G)$ is a local base at $0$, then $h_{top}(\phi)=\sup\{H_{top}(\phi,U\}:U\in\mathcal B\}$.
\item[(b)] If $\mathcal B\subseteq \mathcal C(G)$ is cofinal, then $h_{alg}(\phi)=\sup\{H_{alg}(\phi,U):U\in\mathcal B\}$.
\end{itemize}
\end{lemma}

When the locally compact abelian group $G$ is totally disconnected, then $\mathcal B(G)$ is a local base at $0$ by van Dantzig Theorem from \cite{vD}. This fact applies also in Proposition \ref{cofinal} to prove that $\mathcal B(G)$ is cofinal in $\mathcal C(G)$ when the locally compact abelian group $G$ is compactly covered. Note that in this case the connected component $c(G)$ of $G$ is compact, since it is a connected locally compact abelian group which is also compactly covered, so it admits no isomorphic copy of $\R^n$ for positive integers $n$.

\begin{proposition}\label{cofinal}
If $G$ is a compactly covered locally compact abelian group, then $\mathcal B(G)$ is cofinal in $\mathcal C(G)$.
\end{proposition}
\begin{proof}
Assume first that $G$ is totally disconnected and let $K\in\mathcal C(G)$. Since $\mathcal B(G)$ is a local base at $0$ in $G$, there exists $U\in\mathcal B(G)$ such that $U\subseteq K$. Let $\pi:G\to G/U$ be the canonical projection. Since $G/U$ is discrete and $\pi(K)$ is compact, it follows that $\pi(K)$ is finite. Then $F=\langle\pi(K)\rangle$ is a finite subgroup of $G/U$, as the latter group is torsion being compactly covered and discrete. Therefore, $\pi^{-1}(F)\in\mathcal B(G)$ and $\pi^{-1}(F)\supseteq K$. 

Consider now the general case. Then $c(G)$ is compact as noted above, and let $q:G\to G/c(G)$ be the canonical projection. Let $U\in\mathcal C(G)$.
Then replacing $U$ by $U+c(G)\in\mathcal C(G)$, we can assume without loss of generality that $U\supseteq c(G)$. Therefore, $q(U)=q^{-1}(q(U))$ and $q(U)\in\mathcal C(G/c(G))$. By the first part of the proof, there exists $V\in\mathcal B(G/c(G))$ such that $q(U)\subseteq V$. Now $U\subseteq q^{-1}(V)$ and $q^{-1}(V)\in\mathcal B(G)$ since $c(G)$ is compact.
\end{proof}

Lemma \ref{base}, van Dantzig Theorem and Proposition \ref{cofinal} give the following simplified formulas for the computation of the topological and the algebraic entropy.

\begin{theorem}\label{last:-)}
Let $G$ be a locally compact abelian group and $\phi:G\to G$ a continuous endomorphism.
\begin{itemize}
\item[(a)] If $G$ is totally disconnected, then $h_{top}(\phi)=\sup\{H_{top}(\phi,U):U\in\mathcal B(G)\}$.
\item[(b)] If $G$ is compactly covered, then $h_{alg}(\phi)=\sup\{H_{alg}(\phi,U):U\in\mathcal B(G)\}$.
\end{itemize}
\end{theorem}

\section{Background on Pontryagin duality}\label{pd-sec}

We recall now the definitions and properties concerning the Pontryagin duality needed in this paper, more can be found in \cite{DPS,HR,P}.

For a locally compact abelian group $G$, the Pontryagin dual $\widehat G$ of $G$ is the group of all continuous homomorphisms (i.e., characters) $\chi:G\to \mathbb T$, endowed with the compact-open topology. Pontryagin duality Theorem asserts that $G$ is canonically isomorphic to $\widehat{\widehat G}$.  It is worth reminding also that $G$ is finite if and only if $\widehat G$ is finite, and in this case $\widehat G\cong G$; moreover, $G$ is compact precisely when $\widehat G$ is discrete, and under this assumption $G$ is totally disconnected if and only if $\widehat G$ is torsion. Moreover, we use the fact that $\widehat{\R^n}\cong\widehat \R^n\cong \R^n$ for any non-negative integer $n$.

\smallskip
For a subgroup $A$ of $G$, the annihilator of $A$ in $\widehat G$ is $$A^\perp=\{\chi\in\widehat G:\chi(A)=0\},$$ while for a subgroup $B$ of $\widehat G$, the annihilator of $B$ in $G$ is $$B^\top=\{x\in G:\chi(x)=0\ \text{for every }\chi\in B\}.$$ If $A$ is a closed subgroup of $G$, then 
\begin{equation}\label{perptop}
(A^\perp)^\top=A;
\end{equation}
analogously, if $B$ is a closed subgroup of $\widehat G$, then $(B^\top)^\perp=B$. Moreover, assume that $A$ is a closed subgroup of $G$; it is a standard fact of Pontryagin duality that
\begin{equation}\label{iso-eq}
\widehat A\cong G/A^\perp\ \text{and}\ \widehat{G/A}\cong A^\perp;
\end{equation}
we use these topological isomorphisms several times hereinafter, as well as the following corresponding exact sequences:
$$
0\to A\to G \to G/A \to 0\ \text{and}\ 0\leftarrow G/A^\perp \leftarrow\widehat G\leftarrow A^\perp\leftarrow 0.
$$

The useful observation contained in Lemma \ref{A/B} is based on the standard properties of Pontryagin duality.

\begin{lemma}\label{A/B}
Let $G$ be a locally compact abelian group, and let $B\subseteq A$ be closed subgroups of $G$. Then $\widehat{A/B}$ is topologically isomorphic to $B^\perp/A^\perp$.
\end{lemma}
\begin{proof} 
Consider the embedding $A/B\to G/B$. By Pontryagin duality there exists an open continuous homomorphism $\widehat{G/B}\to \widehat{A/B}$. Since $\widehat{G/B}$ is topologically isomorphic to $B^\perp$, consider the composition $\varphi:B^\perp\to \widehat{A/B}$. Since $\ker\varphi=A^\perp$, the conclusion follows.
\end{proof}

The following lemma gives a useful correspondence between elements of $\mathcal B(G)$ and of $\mathcal B(\widehat G)$.

\begin{lemma}\label{B(G)}
Let $G$ be a locally compact abelian group. Then
\begin{itemize}
\item[(a)] $U\in\mathcal B(G)$ if and only if $U^\perp\in\mathcal B(\widehat G)$.
\item[(b)]If $U_1,\ldots,U_n\in\mathcal B(G)$, then 
$$\left(\sum_{i=1}^nU_i\right)^\perp\cong \bigcap_{i=1}^nU_i^\perp\ \text{and}\ \left(\bigcap_{i=1}^nU_i\right)^\perp\cong\sum_{i=1}^n U_i^\perp.$$
\end{itemize}
\end{lemma}
\begin{proof}
(a)  Since $U$ is open in $G$, the quotient $G/U$ is discrete, therefore $U^\perp\cong \widehat{G/U}$ is compact. Moreover, $U$ is compact, so $G/U^\perp\cong\widehat U$ is discrete, and hence $U^\perp$ is open in $G$.

(c) is a well-known fact in Pontryagin duality.
\end{proof}

It is worth to recall also the following known fact from Pontryagin duality concerning a continuous endomorphism and its dual.

\begin{lemma}
Let $G$ be a locally compact abelian group, $\phi:G\to G$ a continuous endomorphism and $H$ a closed subgroup of $G$. Then:
\begin{itemize}
\item[(a)] $H$ is $\phi$-invariant if and only if $H^\perp$ is $\widehat\phi$-invariant;
\item[(b)] $H$ is fully invariant in $G$ if and only if $H^\perp$ is fully invariant in $\widehat G$.
\end{itemize}
\end{lemma}

Furthermore, we have the following pair of corresponding commutative diagrams, where we let $\psi=\widehat\phi$, and $\overline\phi:G/H\to G/H$ and $\overline\psi:\widehat G/H^\perp\to \widehat G/H^\perp$ are the continuous endomorphisms induced respectively by $\phi$ and $\psi$:
\begin{equation*}\label{diagram}
\xymatrix{
H\ar[d]_{\phi\!\restriction_H}\ar@{^{(}->}[r] & G \ar[d]^\phi\ar@{->>}[r] & G/H\ar[d]^{\overline\phi} & & \widehat G/H^\perp & \widehat G\ar@{->>}[l] & H^\perp \ar@{_{(}->}[l] \\
H\ar@{^{(}->}[r] & G\ar@{->>}[r] & G/H & & \widehat G/H^\perp\ar[u]^{\overline\psi} & \widehat G \ar[u]^\psi\ar@{->>}[l] & H^\perp\ar[u]_{\psi\!\restriction_{H^\perp}} \ar@{_{(}->}[l]
}
\end{equation*}
The second diagram is obtained from the first one by applying the Pontryagin duality functor, where we always take into account the isomorphisms in \eqref{iso-eq}. In particular, 
\begin{equation}\label{3.3}
\begin{split}
\widehat{\overline\psi}\ \text{is conjugated to}\ \phi\!\restriction_H\ \text{and}\ \widehat{\psi\!\restriction_{H^\perp}}\ \text{is conjugated to}\ \overline \phi;\\
\overline\psi\ \text{is conjugated to}\ \widehat{\phi\!\restriction_H}\ \text{and}\ \psi\!\restriction_{H^\perp}\ \text{is conjugated to}\ \widehat{\overline \phi}.
\end{split}
\end{equation}
We say that a continuous endomorphism $\eta:G\to G$ is conjugated to another continuous endomorphism $\gamma:H\to H$ if there exists a topological isomorphism $\xi:G\to H$ such that $\gamma=\xi\eta\xi^{-1}$.

\section{Proof of the Bridge Theorem}\label{bt-sec}

Using the properties of Pontryagin duality recalled in the previous section, we start now with steps taking to the proof of Theorem \ref{BT}.

\begin{proposition}\label{perp}
Let $G$ be a locally compact abelian group, $\phi:G\to G$ a continuous endomorphism and $U\in\mathcal B(G)$.
For every integer $n\geq0$, $$(\phi^{-n}(U))^\perp=(\widehat\phi)^n (U^\perp).$$
\end{proposition}
\begin{proof}
We prove the result for $n=1$, that is, 
\begin{equation}\label{1eq}
(\phi^{-1}(U))^\perp=\widehat\phi (U^\perp). 
\end{equation}
The proof for $n>1$ follows easily from this case noting that $(\widehat\phi)^n=\widehat{(\phi^n)}$. 

Let $V=U^\perp$; then $V\in\mathcal B(\widehat G)$ by Lemma \ref{B(G)}(a) and $U=V^\top$ by \eqref{perptop}. We prove that 
\begin{equation}\label{2eq}
\phi^{-1}(V^\top)=(\widehat\phi(V))^\top,
\end{equation} 
that is equivalent to \eqref{1eq} by \eqref{perptop}. So let $x\in \phi^{-1}(V^\top)$; equivalently, $\phi(x)\in V^\top$, that is $\chi(\phi(x))=0$ for every $\chi\in V$. This occurs precisely when $\widehat\phi(\chi)(x)=0$ for every $\chi\in V$, if and only if $x\in (\widehat\phi(V))^\top$. This chain of equivalences proves \eqref{2eq}.
%
%
\end{proof}

Lemma \ref{B(G)}(b) and Proposition \ref{perp} have the following immediate consequence on the relation between cotrajectories in $G$ and trajectories in $\widehat G$.

\begin{corollary}\label{perp-c}
Let $G$ be a locally compact abelian group, $\phi:G\to G$ a continuous endomorphism and $U\in\mathcal B(G)$.
Then, for every integer $n>0$,  $$C_n(\phi,U)^\perp=T_n(\widehat\phi,U^\perp).$$
\end{corollary}

So now we can see the relation between indexes, with respect to the formulas in Fact \ref{no-mu}.

\begin{proposition}\label{C_n-T_n}
Let $G$ be a locally compact abelian group, $\phi:G\to G$ a continuous endomorphism and $U\in\mathcal B(G)$. Then, for every integer $n>0$,
$$[U:C_n(\phi,U)]=[T_n(\widehat\phi,U^\perp):U^\perp].$$
\end{proposition}
\begin{proof} 
Let $U\in\mathcal B(G)$. Then $U^\perp\in\mathcal B(\widehat G)$ by Lemma \ref{B(G)}(a). By Corollary \ref{perp-c}, $C_n(\phi,U)^\perp=T_n(\widehat \phi,U^\perp)$ for every integer $n>0$.
Consider the following sequences:
$$0\to C_n(\phi,U)\to U\to G$$
and
$$\widehat G=0^\perp\leftarrow T_n(\widehat\phi,U^\perp)=C_n(\phi,U)^\perp\leftarrow U^\perp\leftarrow 0=G^\perp.$$
By Lemma \ref{A/B}, $\widehat{U/C_n(\phi,U)}\cong T_n(\widehat\phi,U^\perp)/U^\perp$. Since these quotients are finite, we have the isomorphism $U/C_n(\phi,U)\cong T_n(\widehat\phi,U^\perp)/U^\perp$, and this concludes the proof.
\end{proof}

The following result gives ``locally'' the exact relation between the topological entropy of a continuous endomorphism $\phi$ and the algebraic entropy of its dual $\widehat\phi$.

\begin{proposition}\label{C_n-T_n*}
Let $G$ be a locally compact abelian group, let $\phi:G\to G$ be a continuous endomorphism and $U\in\mathcal B(G)$. Then
$$H_{top}(\phi,U)=H_{alg}(\widehat\phi,U^\perp).$$
\end{proposition}
\begin{proof}
The conclusion follows from Proposition \ref{C_n-T_n}, applying Fact \ref{no-mu}.
\end{proof}

We are now in position to prove our main theorem, that is Theorem \ref{BT}:

\begin{theorem}
Let $G$ be a totally disconnected locally compact abelian group and let $\phi:G\to G$ be a continuous endomorphism. Then $h_{top}(\phi)=h_{alg}(\widehat\phi)$.
\end{theorem}
\begin{proof}
Theorem \ref{last:-)} gives 
$$
h_{top}(\phi)=\sup\{H_{top}(\phi,U):U\in\mathcal B(G)\}\ \text{and}\ h_{alg}(\widehat \phi)=\sup\{H_{alg}(\widehat\phi,K):K\in\mathcal B(\widehat G)\}.
$$
By Proposition \ref{C_n-T_n*}, $H_{top}(\phi,U)=H_{alg}(\widehat\phi,U^\perp)$ for every $U\in\mathcal B(G)$, so $h_{top}(\phi)\leq h_{alg}(\widehat \phi)$.
If $K\in\mathcal B(\widehat G)$, then $K=U^\perp$, where $U=K^\bot\in\mathcal B(G)$ by Lemma \ref{B(G)}. So Proposition \ref{C_n-T_n*} applies again to give that $H_{top}(\phi,K^\top)=H_{alg}(\widehat\phi,K)$ for every $K\in\mathcal B(\widehat G)$, hence $h_{top}(\phi)\geq h_{alg}(\widehat \phi)$, and this concludes the proof.
\end{proof}

\section{Additivity implies the  Bridge Theorem}\label{at-sec}

In the next example we see that the Bridge Theorem holds for continuous endomorphisms of $\R^n$.

\begin{example}\label{Rn}
Let $n$ be a positive integer and $\phi:\R^n\to \R^n$ a continuous endomorphism. Then $h_{top}(\phi)=h_{alg}(\widehat\phi)$.

Indeed, $\phi$ is $\R$-linear so it is given by an $n\times n$ matrix $M_\phi$. The same argument applies to $\widehat\phi:\R^n\to \R^n$, moreover $M_{\widehat\phi}$ is the transposed matrix of $M_\phi$ by definition of the dual endomorphism. In particular, $\phi$ and $\widehat\phi$ have the same eigenvalues, say $\{\lambda_1,\ldots,\lambda_n\}\subseteq\mathbb C$.
Bowen proved in \cite{B} that
$$h_{top}(\phi)=\sum_{|\lambda_i|>1}\log|\lambda_i|$$ 
and Virili established in \cite{V} that
$$h_{alg}(\widehat\phi)=\sum_{|\lambda_i|>1}\log|\lambda_i|;$$ 
hence, $h_{top}(\phi)=h_{alg}(\widehat\phi)$.
\end{example}

Before giving the proof of Theorem \ref{AT->BT} we describe in detail the fully invariant subgroups that we use there. So let $G$ be a locally compact abelian group. The connected component $c(G)$ of $G$ is fully invariant. Another fully invariant subgroup of $G$ that is natural to consider in this setting is the subgroup $B(G)$ consisting of all elements of $G$ contained in some compact subgroup of $G$, equivalently, $$B(G)=\sum\{K\leq G: K\ \text{compact}\}.$$ Clearly, also $B(G)$ is fully invariant in $G$. Moreover, note that $G$ is compactly covered precisely when $G=B(G)$.

These two subgroups $c(G)$ and $B(G)$ are related by Pontryagin duality in the sense that 
$$B(G)^\perp=c(\widehat G)\ \text{and}\ c(G)^\perp=B(\widehat G).$$
Indeed, $B(G)=\sum\{K\leq G: K\ \text{compact}\}$, so $B(G)^\perp=\bigcap\{K^\perp:K\leq G\ \text{compact}\}$; since $K$ is compact, $\widehat K\cong \widehat G/K^\perp$ is discrete, hence $K^\perp$ is an open subgroup of $\widehat G$, and now it suffices to recall that $c(\widehat G)=\bigcap\{U\leq \widehat G:U\ \text{open}\}$.

Consider their intersection $c(B(G))=c(G)\cap B(G)$ and their sum $G_1=c(G)+B(G)$, which are other fully invariant subgroups of $G$. Note that the annihilators of $c(B(G))$ and $G_1$ in $\widehat G$ are respectively $c(B(G))^\perp=c(\widehat G)+B(\widehat G)$ and $G_1^\perp=c(B(\widehat G))$.

\begin{remark}
Recall that a locally compact abelian group $G$ can always be written as $G=\R^n\times G_0$, for some positive integer $n$ and a locally compact abelian group $G_0$ admitting a compact open subgroup $K$. Then $c(B(G))=c(K)$, and $c(G)=\R^n\times c(K)$.
Moreover, $G_1=\R^n\times B(G)$, while $c(B(G))^\perp=\R^n\times B(\widehat G)$ and $G_1^\perp=c(B(\widehat G))$.
\end{remark}

The situation is described by the following diagrams, each one dual with respect to the other.
$$
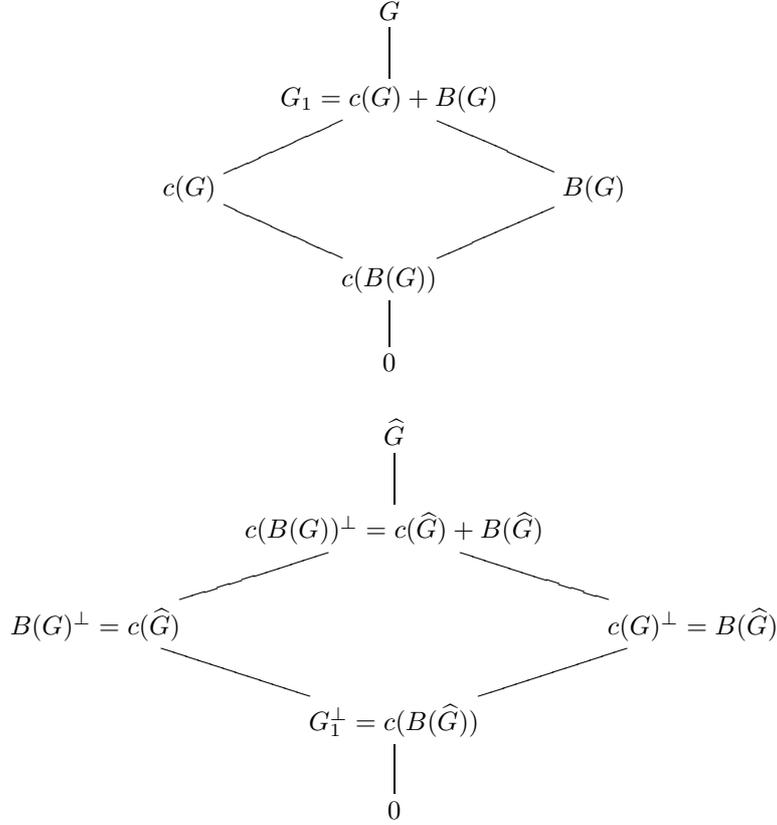
\xymatrix@-0.5pc{
& G \ar@{-}[d] &  \\
& G_1=c(G)+B(G) \ar@{-}[dl]\ar@{-}[dr] & \\
c(G)\ar@{-}[dr] & & B(G)\ar@{-}[dl] \\
& c(B(G))\ar@{-}[d] & \\
& 0 & 
}$$

$$\xymatrix@-0.5pc{
& \widehat G\ar@{-}[d] & \\
& c(B(G))^\perp=c(\widehat G)+B(\widehat G)\ar@{-}[dl]\ar@{-}[dr] &\\
B(G)^\perp=c(\widehat G)\ar@{-}[dr] & & c(G)^\perp=B(\widehat G)\ar@{-}[dl] \\
& G_1^\perp=c(B(\widehat G))\ar@{-}[d] & \\
& 0 &
}$$

\medskip
Invariance under conjugation was proved for the topological and the algebraic entropy respectively in \cite{B} and \cite{V}; namely,
if $G$ is a locally compact abelian group, $\phi:G\to G$ a continuos endomorphism, $H$ is another locally compact abelian group and $\xi:G\to H$ a topological isomorphism, then $$h_{top}(\phi) = h_{top}(\xi\phi\xi^{-1})\ \text{and}\ h_{alg}(\phi) = h_{alg}(\xi\phi\xi^{-1}).$$

So, in view of \eqref{3.3}, we have the following lemma needed in the proof of Theorem \ref{AT->BT}.

\begin{lemma}\label{pontr-rem}
Let $G$ be a locally compact abelian group, $\phi:G\to G$ a continuous endomorphism and $H$ a $\phi$-invariant closed subgroup of $G$. Let $\psi=\widehat\phi$, and $\overline \phi:G/H\to G/H$ and $\overline\psi:\widehat G/H^\perp\to \widehat G/H^\perp$ the continuous endomorphisms induced respectively by $\phi$ and $\psi$. Then:
\begin{itemize}
\item[(a)] $h_{top}(\widehat{\phi\!\restriction_H})=h_{top}({\overline\psi})$ and $h_{top}(\widehat{\overline\phi})=h_{top}({\psi\!\restriction_{H^\perp}})$;
\item[(b)] $h_{alg}(\widehat{\phi\!\restriction_H})=h_{alg}({\overline\psi})$ and $h_{alg}(\widehat{\overline\phi})=h_{alg}({\psi\!\restriction_{H^\perp}})$.
\end{itemize}
\end{lemma}

We are now in position to give the proof of Theorem \ref{AT->BT}, which is based on the scheme of fully invariant subgroups of $G$ just described, also in the diagrams:

\begin{theorem}
Assume that $\FIAT_{top}$ and $\FIAT_{alg}$ hold. If $G$ is a locally compact abelian group and $\phi:G\to G$ a continuous endomorphism, then $h_{top}(\phi)=h_{alg}(\widehat\phi)$.
\end{theorem}
\begin{proof}
Let $\psi=\widehat\phi$.

First we prove the conclusion assuming that $G=\R^n\times K$, where $n$ is a non-negative integer and $K$ a compact abelian group. In this case $\widehat G=\R^n\times X$, where $X=\widehat K$ is a discrete abelian group. Then $c(\widehat G)=\R^n$ and so it is fully invariant in $\widehat G$; equivalently, $K$ is fully invariant in $G$. Let $\overline\phi:\R^n\to \R^n$ and $\overline\psi:X\to X$ be the continuous endomorphisms induced by $\phi$ and $\psi$ respectively. The situation is described by the following pair of corresponding diagrams.
$$
\xymatrix{
K\ar[d]_{\phi\!\restriction_{K}}\ar@{^{(}->}[r] & G \ar[d]^\phi\ar@{->>}[r] & \R^n\ar[d]^{\overline\phi} & & X & \widehat G\ar@{->>}[l] & \R^n \ar@{_{(}->}[l] \\
K\ar@{^{(}->}[r] & G\ar@{->>}[r] & \R^n & & X\ar[u]^{\overline\psi} & \widehat G \ar[u]^\psi\ar@{->>}[l] & \R^n\ar[u]_{\psi\!\restriction_{\R^n}} \ar@{_{(}->}[l]
}
$$
Now $h_{top}(\phi\restriction_K)=h_{alg}(\overline\psi)$ by Theorem \ref{DG-bt} and $h_{top}(\overline\phi)=h_{alg}(\psi\restriction_{\R^n})$ by Example \ref{Rn}, where these two equalities hold also in view of Lemma \ref{pontr-rem}.
By $\FIAT_{top}$ and $\FIAT_{alg}$ we can conclude that $h_{top}(\phi)=h_{alg}(\psi)$, as required.

\smallskip
We pass now to the general case. Since $\widehat G$ is a locally compact abelian group, $\widehat G=\R^n\times G_0$, for some non-negative integer $n$ and some locally compact abelian group $G_0$ admitting a compact open subgroup $K$. Consider $B(\widehat G)$, which is compactly covered and fully invariant in $\widehat G$. Moreover, $B(\widehat G)=B(G_0)$ and $B(\widehat G_0)$ is open in $G_0$ as it contains the compact open subgroup $K$ of $G_0$; in particular, the quotient $X=G_0/B(G_0)$ is discrete and $\widehat G/B(\widehat G)=\R^n\times X$. It is worth to note that $c(G)^\perp=B(\widehat G)$ so that $\widehat G/B(\widehat G)\cong\R^n\times X$.
Let $\overline\phi:G/c(G)\to G/c(G)$ be the continuous endomorphism induced by $\phi$. Moreover, let $\overline\psi:\R^n\times X\to \R^n\times X$ be the continuous endomorphism conjugated to the one induced by $\psi$ on $\widehat G/B(\widehat G)\to \widehat G/B(\widehat G)$; this can be done without loss of generality in view of the previously described isomorphism between $\widehat G/B(\widehat G)$ and $\R^n\times X$ taking into account Lemma \ref{pontr-rem}. The situation is described by the following corresponding diagrams.
$$\xymatrix{
c(G)\ar[d]_{\phi\!\restriction_{c(G)}}\ar@{^{(}->}[r] & G \ar[d]^\phi\ar@{->>}[r] & G/c(G)\ar[d]^{\overline\phi} & & \R^n\times X & \widehat G\ar@{->>}[l] & B(\widehat G)\ar@{_{(}->}[l] \\
c(G)\ar@{^{(}->}[r] & G\ar@{->>}[r] & G/c(G) & & \R^n\times X\ar[u]^{\overline\psi} & \widehat G \ar[u]^\psi\ar@{->>}[l] & B(\widehat G) \ar[u]_{\psi\!\restriction_{B(\widehat G)}} \ar@{_{(}->}[l]
}$$
Now $c(G)\cong \R^n \times c(K)$, so $h_{top}(\phi\restriction_{c(G)})=h_{alg}(\overline\psi)$ by the first step of the proof, while $h_{top}(\overline\phi)=h_{alg}(\psi\restriction_{B(\widehat G)})$ by Theorem \ref{BT}, where these two equalities hold also in view of Lemma \ref{pontr-rem}. Hence, $\FIAT_{top}$ and $\FIAT_{alg}$ yield $h_{top}(\phi)=h_{alg}(\psi)$.
\end{proof}

\begin{remark}\label{LastRemark} 
We apply $\FIAT_{top}$ and $\FIAT_{alg}$ in the above proof only with respect to the fully invariant subgroups $c(G)$ and $B(G)$.
More precisely, we need  $\FIAT_{top}$ with respect to $c(G)$, while $\FIAT_{top}$
with respect to $B(G)$ is needed only in the class of connected locally compact abelian groups.
Analogously, $\FIAT_{alg}$ is needed with respect to $B(G)$, while $\FIAT_{alg}$
with respect to $c(G)$ is needed only in the class of locally compact abelian groups $G$ with trivial $B(G)$ (i.e., $G \cong \R^n\times X$, where $X$ is discrete and torsion-free).
\end{remark}

\end{document}